\documentclass[letterpaper, 10 pt, conference]{ieeeconf}
\IEEEoverridecommandlockouts

\makeatletter

\let\proof\@undefined
\let\endproof\@undefined
\makeatother

\usepackage{graphicx}          
\usepackage{amssymb}
\usepackage{amsmath}
\usepackage{amsthm}
\usepackage{cleveref}
\usepackage{pgfplots}
\usepackage{array}
\usepgfplotslibrary{fillbetween}
\pgfplotsset{compat=newest}
\usetikzlibrary{calc}

\crefname{equation}{}{}
\Crefname{equation}{Equation}{Equations}

\newtheorem{thm}{Theorem}
\crefname{thm}{Theorem}{Theorems}
\Crefname{thm}{Theorem}{Theorems}

\newtheorem{assum}{Assumption}
\crefname{assum}{Assumption}{Assumptions}
\Crefname{assum}{Assumption}{Assumptions}
\newtheorem{prop}{Proposition}

\crefname{figure}{Figure}{Figures}
\Crefname{figure}{Figure}{Figures}

\theoremstyle{definition}
\usepackage{thmtools}

\crefname{definition}{Definition}{Definitions}
\Crefname{definition}{Definition}{Definitions}

\theoremstyle{plain}
\declaretheorem[style=definition,qed=$\blacksquare$]{example}

\crefname{example}{Example}{Examples}
\Crefname{example}{Example}{Examples}

\newtheorem{corollary}{Corollary}
\theoremstyle{remark}

\theoremstyle{definition}

\newcommand{\ul}{\underline}

\newcommand{\F}{\mathcal{F}}
\newcommand{\Rsim}{\overset{\mathcal{R}}{\sim}}

\newcommand{\into}{\rightarrow}
\newcommand{\Rbb}{\mathbb{R}}

\newcommand{\tr}{\top}

\newcommand{\sm}{\backslash}
\newcommand{\set}[1]{\left\{#1\right\}}

\newcommand{\paren}[1]{\left( #1 \right)}
\newcommand{\brak}[1]{\left[ #1 \right]}
\newcommand{\mat}[2]{\brak{\begin{array}{#1} #2 \end{array}}}

\newcommand{\eqnn}[1]{\begin{equation}\begin{aligned} #1 \end{aligned}\end{equation}}

\newcommand{\spec}{\operatorname{spec}}
\newcommand{\id}{\operatorname{id}}

\usepackage[textwidth=2in]{todonotes}

\definecolor{CornflowerBlue}{rgb}{0.258824,0.258824,0.435294}
\definecolor{cfblue}{rgb}{0.258824,0.258824,0.435294}
\definecolor{SkyBlue}{rgb}{0.196078,0.6,0.8}
\definecolor{TitleBlue}{RGB}{47, 67, 114}
\definecolor{dblue}{rgb}{.098,.243,.424}
\definecolor{lblue}{rgb}{.33,.57,.835}
\definecolor{llblue}{rgb}{.447,.643,.831}
\definecolor{lbluesam}{rgb}{.447,.643,.831} 
\definecolor{mblue}{rgb}{0.176, 0.380, 0.659}
\definecolor{lcomp}{rgb}{.969,.765,.416}
\definecolor{ddorange}{rgb}{0.624, 0.365, 0}
\definecolor{dorange}{rgb}{0.72, 0.506, 0.125}
\definecolor{lorange}{rgb}{0.961, 0.678, 0.165}
\definecolor{lgreen}{rgb}{.812,.969,.435}
\definecolor{dgreen}{RGB}{15,111,3}
\definecolor{mgreen}{RGB}{84, 174, 50}
\definecolor{lyellow}{rgb}{1,.859,.451}
\definecolor{dyellow}{rgb}{.651,.482,0}
\definecolor{lred}{rgb}{1,.6,.451}
\definecolor{dred}{rgb}{.65,.176,0}
\definecolor{dcompb}{RGB}{157,35,0}  
\definecolor{lcompb}{RGB}{186,70,30}
\definecolor{llcompb}{RGB}{255,136,92}
\definecolor{lcompbsam}{RGB}{255,136,92}  
\definecolor{dpurple}{RGB}{45,0,95}
\definecolor{mpurple}{RGB}{77,0,159}
\definecolor{lpurple}{RGB}{143,73,206}
\definecolor{purplea}{RGB}{122,24,207}
\definecolor{purpleb}{RGB}{68,17,112}
\definecolor{pptgreen}{RGB}{147, 208, 80}

\title{Generalizing infinitesimal contraction analysis to hybrid systems} 
\author{Samuel A. Burden and Samuel D. Coogan\thanks{S. A. Burden is with the Department of Electrical Engineering, University of Washington, Seattle, WA, USA, \texttt{sburden@uw.edu}. S. D. Coogan is with the School of Electrical and Computer Engineering and the School of Civil and Environmental Engineering, Georgia Institute of Technology, Atlanta, GA, USA, \texttt{sam.coogan@gatech.edu}.  Support is provided by the U.~S. Army Research Laboratory and the U.~S. Army Research Office under contract/grant number W911NF-16-1-0158 and by the National Science Foundation under award number 1749357.
}}

\date{}

\renewcommand\footnotemark{}

\begin{document}
\maketitle
\begin{abstract}
Infinitesimal contraction analysis, 
wherein global asymptotic convergence results are obtained from local dynamical properties,
has proven to be a powerful tool for applications in biological, mechanical, and transportation systems. 
Thus far, the technique has been restricted to systems governed by a single smooth differential or difference equation.
We generalize infinitesimal contraction analysis to hybrid systems governed by interacting differential and difference equations.
Our theoretical results are illustrated on a series of examples.

\end{abstract}

\renewcommand{\H}{\mathcal{H}}
\newcommand{\D}{\mathcal{D}}
\renewcommand{\F}{\mathcal{F}}
\newcommand{\G}{\mathcal{G}}
\newcommand{\R}{\mathcal{R}}
\newcommand{\g}{\mathfrak{g}}
\newcommand{\N}{\mathcal{N}}
\newcommand{\mcM}{\mathcal{M}}
\newcommand{\olR}{\overline{\R}}
\newcommand{\xcap}{x^\text{cap}}
\newcommand{\xjam}{x^\text{jam}}
\newcommand{\xcrit}{x^\text{crit}}
\newcommand{\ulx}{\underline{x}}
\newcommand{\J}{\mathcal{J}}

\section{Introduction}

A dynamical system is \emph{contractive} if all trajectories converge to one another~\cite{LOHMILLER:1998bf}.
Contractive systems enjoy strong asymptotic properties, e.g. any equilibrium or periodic orbit is globally asymptotically stable.
Provocatively, these global results can sometimes be obtained by analyzing local (or \emph{infinitesimal}) properties of the system's dynamics.
In smooth differential (or difference) equations, for instance, a bound on a matrix measure (or induced norm) of the derivative of the equation can be used to prove global contractivity~\cite{LOHMILLER:1998bf, Pavlov:2004lr, Sontag:2010fk, Sontag:2014eu, Margaliot:2015wd, Aminzarey:2014rm};
this approach has been successfully applied to biological~\cite{Raveh:2015wm, Margaliot:2014qv,Wang2005-zj}, mechanical~\cite{Lohmiller2000-kj,Manchester2014-kc}, and transportation~\cite{coogan2015compartmental, Como:2015ne} systems. 

At its core, the \emph{infinitesimal} approach to contractivity leverages local dynamical properties of continuous--time flow (or discrete--time reset) to bound the time rate of change of the distance between trajectories.
We generalize infinitesimal contraction analysis to hybrid systems, leveraging local dynamical properties of continuous--time flow \emph{and} discrete--time reset to bound the time rate of change of the \emph{intrinsic} distance between trajectories.
The intrinsic distance metric we employ is defined in a natural way based on the idea that the distance between a point in a guard and the point it resets to should be zero.%
\footnote{i.e. $d(x,\R(x)) = 0$ for all $x\in\G$}

This idea was proven in~\cite{Burden2015-ip} to yield a (pseudo%
\footnote{On the topological quotient space obtained from the smallest equivalence relation on $\D$ containing $\set{(x,\R(x)) : x\in\G}$, the function is a distance metric compatible with the quotient topology~\cite[Thm.~13]{Burden2015-ip}.}%
) distance metric that assigns finite distance to states in distinct discrete modes (so long as there exist trajectories connecting the modes).
The intrinsic distance metric is distinct from the Skorohod~\cite{Gokhman2008-br} or Tavernini~\cite{Tavernini1987-wu} \emph{trajectory metrics}~\cite[Sec.~V-A]{Burden2015-ip} and from the \emph{distance function} introduced in~\cite{Biemond2016-ur}; it is a particular instantiation of the \emph{distance function} defined in~\cite{Biemond2013-ph}.

The conditions we obtain for infinitesimal contraction have intuitive appeal:
the derivative of the vector field, which captures the infinitesimal dynamics of continuous--time flow, must be infinitesimally contractive with respect to the matrix measure determined by the vector norm used in each discrete mode;
the \emph{saltation  matrix}, which captures the infinitesimal dynamics of discrete--time reset, must be contractive with respect to the induced norm determined by the vector norms used on either side of the reset.
If upper and lower bounds on \emph{dwell time} are available, we can bound the intrinsic distance between trajectories, whether this distance is expanding or contracting in continuous-- or discrete-- time.
We present several examples to illustrate these theoretical contributions.

\section{Notation}

Given a collection of sets $\set{S_\alpha}_{\alpha\in A}$ indexed by $A$, the \emph{disjoint union} of the collection is defined $\coprod_{\alpha\in A} S_\alpha=\cup_{\alpha\in A}(\{\alpha\}\times S_\alpha)$. 
Given $(a,x)\in\coprod_{\alpha\in A}S_\alpha$, we will simply write $x\in\coprod_{\alpha\in A}S_\alpha$ when $a$ is clear from context. 
For a function $\gamma$ with scalar argument, we denote limits from the left and right by $\gamma(t^-)=\lim_{\sigma\uparrow t}\gamma(\sigma)$ and $\gamma(t^+)=\lim_{\sigma\downarrow t}\gamma(\sigma)$.
Given a smooth function $f : X \times Y \into Z$, 
we let 
$D_xf : TX \times Y \into TZ$
denote the 
\emph{derivative of $f$ with respect to $x\in X$}
and 
$Df = (D_x f, D_y f) : TX \times TY \into TZ$ 
denote the derivative of $f$ with respect to both $x\in X$ and $y\in Y$.
Here, $TX$ denotes the \emph{tangent bundle} of $X$; when $X\subset\Rbb^d$ we have $TX = X\times\Rbb^d$.
The \emph{induced norm} of a linear function $M:\Rbb^{n_j}\into\Rbb^{n_{j'}}$~is 
\begin{align}
  \label{eq:46}
  \|M\|_{j,j'}=\sup_{x\in \mathbb{R}^{n_{j'}}}\frac{|Mx|_{j}}{|x|_{j'}}
\end{align}
where $|\cdot|_j$ and $|\cdot|_{j'}$ denote the vector norms on $\mathbb{R}^{n_j}$ and $\mathbb{R}^{n_{j'}}$, respectively; 
when the norms are clear from context, we omit the subscripts. 
The \emph{matrix measure} of $A\in\Rbb^{n\times n}$, denoted $\mu(A)$, is 
\begin{align}
  \label{eq:51}
  \mu(A)=\lim_{h\downarrow 0}\frac{(\|I+hA\|-1)}{h}.
\end{align}

\section{Preliminaries}
A \emph{hybrid system} is a tuple $\H=(\D,\F,\G,\R)$ where:
\begin{itemize}
\item[$\D$] $=\coprod_{j\in \J}\D_j$ 
  is a set of states
  where 
  $\J$ is a finite index set 
  and 
  $\D_j=\mathbb{R}^{n_j}$ is equipped with a norm $|\cdot|_j$ for some $n_j\in\mathbb{N}$ in each domain $j\in\J$;
\item[$\F$] $:[0,\infty)\times\D\to T\D$ is a vector field that we interpret as 
$\F_j=\left.\F\right|_{\D_j}:[0,\infty)\times \D_j\to\mathbb{R}^{n_j}$ for each $j\in\J$;
\item[$\G$] $=\coprod_{j\in \J} \G_j$ is a guard set with $\G_j\subset \D_j$ for all $j\in\J$;
\item[$\R$] $:\G\to\D$ is a reset map.
\end{itemize}
We have assumed that $\D_j=\mathbb{R}^{n_j}$ for all $j\in\J$ for ease of exposition. 
In practice, the domains of the hybrid system may be restricted to subsets of Euclidean space, as in the examples below. 
In full generality, hybrid systems exhibit a wide range of behaviors;
our theoretical results require that we restrict the class of hybrid systems under consideration to those satisfying the following assumptions.
First, we assume that the guard does not intersect the image of the reset to preclude multiple simultaneous discrete transitions.
\begin{assum}[isolated discrete transitions]  
\label{assum:reset}
$  \R(\G)\cap\G = \emptyset.$
\end{assum}

Informally, an {execution} of a hybrid system is a 
right--continuous function of time that satisfies the continuous--time dynamics specified by $\F$ and the discrete--time dynamics specified by $\G$ and $\R$;
Formally, a function $\chi:[\tau,\infty)\to \D$ with $\tau\geq 0$ is an \emph{execution} of the hybrid system if:
\begin{enumerate}
\item $  D\chi(t)=\F(t,\chi(t))$ for almost all $t\geq \tau$;
\item $\chi(t^+)=\chi(t)$ for all $t\geq \tau$;
\item $\chi(t^-)= \chi(t)$ if and only if $\chi(t)\not\in \G$;
\item Whenever $\chi(t^-)\neq \chi(t)$, then $\chi(t^-)\in\G$ and $\chi(t)=\R(\chi(t^-))$.
\end{enumerate}

\begin{assum}[existence and uniqueness]
\label{assum:unique}
For any initial condition $x\in \D\backslash \G$ and any initial time $\tau\geq 0$, there exists a unique execution $\chi:[\tau,\infty)\into\Rbb$ satisfying $\chi(\tau)=x$; for $x\in \G$, we consider the execution from $x$ to be the unique execution initialized at $\R(x)\in\D\sm\G$.
\end{assum}

\noindent
Assumption \ref{assum:unique} implies that for each $x\in\D$ and all $\tau\geq 0$ there exists exactly one execution $\chi$ initialized at $\chi(x) = \tau$ that exists for all time $t \ge \tau$, 
that is, 
the hybrid system is 
deterministic,
nonblocking,
and does not exhibit finite escape time. 
We denote this unique execution as $\phi(t,\tau,x)=\chi(t)$; 
 the function $\phi$ defined in this way is the \emph{flow} of the hybrid system;
we adopt the convention that $\phi(\tau^-,\tau,x)=x$ for all $x\in \D$ and all $\tau\in\mathbb{R}$. 
If $\phi(\sigma,\tau,x)\in \D_j$ for all $\sigma\in[\tau,t]$
for some $t\geq \tau$, $j\in\J$, and $x\in\D_j$, 
we write $\phi_j(t,\tau,x)=\phi(t,\tau,x)$ to emphasize that the execution restricted to $[\tau,t]$ lies entirely within $\D_j$. 

\begin{assum}[smooth vector field]
\label{assump:vf}
For all $j\in \J$, $\F_j$ is a smooth vector field.
\end{assum}

\begin{assum}[no Zeno executions]
\label{assum:exec}
No execution undergoes an infinite number of resets in finite time.
\end{assum}

\noindent
We define $\G_{j,j'}=\R^{-1} (\D_{j'})\cap \G_j$ 
for each $j,j'\in\J$ 
and make the following assumptions on guards and resets.
\begin{assum}[differentiability of guards and resets]
\label{assump:guard}
For each $j,j'\in\J$, whenever $\G_{j,j'}\neq \emptyset$, there exists continuously differentiable and nondegenerate%
\footnote{i.e. $Dg_{j,j'}(x) \ne 0$ for all $x\in\D_j$}
$g_{j,j'}:\D_j\to \mathbb{R}$ such that $\G_{j,j'}\subseteq \{x\in \D_j: g_{j,j'}(x)\leq 0\}\subseteq \G_{j}$. Further, there exists continuously differentiable $\R_{j,j'}: \{x\in \D_j: g_{j,j'}(x)\leq 0\}\to \D_{j'}$ such that $\left.\R_{j,j'}\right|_{\G_{j,j'}}=\left. \R \right|_{\G_{j,j'}}$.
\end{assum}

\noindent
In each domain $j\in\J$, we do not allow executions to graze the guard set $\G_j$ and thus impose a transversality property on the vector field $\F_j$. 
\begin{assum} [vector field transverse to guard]
\label{assump:transverse}
  For all $j,j'\in\J$, $t\geq 0$, and $x\in \G_{j,j'}$:
  \begin{equation}
    \label{eq:21}
  Dg_{j,j'}(x)\cdot \F_j(t,x)< 0.
  \end{equation}
\end{assum}

\noindent
The reset map $\R$ induces an equivalence relation $\Rsim$ on $\D$ defined as 
the smallest equivalence relation containing $\{(a,b)\in \G \times \D:\R(a)=b\}\subset \D\times \D$, for which we write $a \Rsim b$ to indicate $a$ and $b$ are related. The equivalence class for $x\in\D$ is defined as $[x]_{\R}=\{y\in \D|x\Rsim y\}$. The \emph{quotient space} induced by the equivalence relation is denoted
\begin{align}
  \label{eq:18}
  \mcM=\{[x]_{\R}|x\in \D\}
\end{align}
endowed with the quotient topology~\cite[Appendix~A]{Lee2012-mb}.

We now consider paths that will be used to define a distance function on the quotient $\mcM$. 
To that end, a \emph{countable partition} of the interval $[0,1]$ is a countable collection $\{r_i\}_{i=0}^k\subset [0,1]$ with possibly $k=\infty$ satisfying $0=r_0\leq r_1\leq r_2\leq \cdots$ and $r_k=1$ if $k$ is finite or $\lim_{i\to \infty}r_i=1$ if $k=\infty$. A path $\gamma:[0,1]\to \D$ is \emph{$\R$-connected} if 
there exists a countable partition $\{r_i\}_{i=0}^k$ of $[0,1]$ such that $\left.\gamma\right|_{[r_i,r_{i+1})}$ is continuous for each $i\in\{0,1,\ldots,k-2\}$ and $\left. \gamma \right |_{[r_{k-1},r_k]}$ is continuous whenever $k$ is finite, and additionally $\lim_{r\uparrow r_{i}} \gamma(r) \Rsim \gamma(r_{i+1})$ for all $i\in\{1,\ldots,k-1\}$. 

Because each section $\left.\gamma\right|_{[r_i,r_{i+1})}$ is continuous, it must necessarily belong to a single $\D_j$ for $j\in\J$.
We further say that $\gamma$ is \emph{piecewise--differentiable} if each section $\left.\gamma\right|_{[r_i,r_{i+1})}$ is piecewise--differentiable. Intuitively, an $\R$-connected path $\gamma$ is a path through the domains $\{\D_j\}_{j\in \J}$ of the hybrid system that jumps through the reset map $\R$ (forward or backward) a countable number of times. With a slight abuse of notation,%
\footnote{Formally, $\pi\circ\gamma$ is a path in $\mcM$, where $\pi:\D\to\mcM$ is the \emph{quotient projection}.}
we consider $\gamma$ a path in $\mcM$. With this identification, all $\R$-connected paths are (more precisely:  descend to) continuous paths in the quotient space $\mcM$.

The length of a piecewise--differentiable path $\gamma_j:[0,1]\to \D_j$ completely contained in a domain $\D_j$ is computed in the usual way using the norm $|\cdot|_j$ in $\D_j$: $L_j(\gamma_j)=\int_0^1|D\gamma_j(r)|_jdr$. We drop the subscript for $L$ when the domain is clear from context and instead write $L(\gamma_j)$. 

Let $\Gamma$ denote the set of  $\R$-connected and piecewise--differentiable paths in $\mcM$, and let 
\begin{align}
  \label{eq:31}
  \Gamma(a,b)=\{\gamma\in\Gamma:\gamma(0)=a\text{ and }\gamma(1)=b,\ a,b\in\D\}.
\end{align}
We use the norm--induced length of each domain to define a \emph{length structure}~\cite[Ch.~2]{Burago2001-sj} on $\mcM$ from which we derive a distance metric.  To that end, for $\gamma\in\Gamma$, let $\{r_i\}_{i=0}^k$ be a countable partition of $[0,1]$ such that $\left.\gamma\right|_{[r_i,r_{i+1})}$ is piecewise--differentiable for all $i\in\{0,1,\ldots,k-1\}$. 
With the length of $\gamma$ defined as
\begin{align}
  \label{eq:19}
  L(\gamma)=\sum_{i=1}^kL(\left.\gamma\right|_{[r_i,r_{i+1})}),
\end{align}
the function $d:\mcM\times \mcM\to \mathbb{R}_{\geq 0}$ defined by
\begin{align}
  \label{eq:20}
  d(x,y)=\inf_{\gamma\in\Gamma(x,y)}L(\gamma)
\end{align}
is a distance metric on $\mcM$ compatible with the quotient topology~\cite[Thm.~13]{Burden2015-ip}.

The final required technical assumption is closely related to continuity of $\phi$ with respect to initial conditions $x$, as claimed in Proposition \ref{prop:cont} below (some proofs omitted due to page constraints).

\begin{assum}[forward--invariance of $\R$-connected paths]
\label{assum:Rconnect}
For all $t\geq \tau$ and all piecewise--differentiable $\R$-connected paths $\gamma\in\Gamma $, $\phi(t,\tau,\gamma(r))$, interpreted as a function of $r$, is a piecewise--differentiable $\R$-connected path.
\end{assum}

\begin{prop}
\label{prop:cont}
Under Assumptions~1--\ref{assum:Rconnect},
the flow $\phi(t,\tau,x)$ varies continuously with respect to $x$ for all $t\geq \tau$ such that $\phi(t^-,\tau,x)\not\in \G$.
\end{prop}

\noindent
It is well--known~\cite{Aizerman1958-ih} that Assumptions~1--\ref{assum:Rconnect} together ensure that the flow $\phi$ is differentiable almost everywhere and, moreover, its derivative can be computed by solving a jump--linear--time--varying differential equation as in the following Proposition.
\begin{prop}
\label{prop:salt}
Under Assumptions~1--\ref{assum:Rconnect},
given an initial time $\tau\geq 0$ 
and 
a piecewise--differentiable $\R$-connected path
$\gamma\in\Gamma$, 
let $\psi(t,r)=\phi(t,\tau,\gamma(r))$ for all $t\ge \tau$  
and define
\begin{align}
  \label{eq:24}
  w(t,r)=D_r\psi(t,r)
\end{align}
whenever the derivative exists. 
Then 
$w(\tau^-,r)=D_r\gamma(r)$ 
and 
$w(\cdot,r)$ satisfies a linear--time--varying differential equation
\begin{align}
\label{eq:variational}
  D_t{w}(t,r)=D_x \F(t,\psi(t,r))w(t,r),\ \psi(t^-,r)\in \D\backslash\G,
\end{align}
with jumps 
\begin{align}
\label{eq:saltation}
  w(t,r)=\Xi(t,\psi(t^-,r)) w(t^-,r),\ \psi(t^-,r)\in \G,
\end{align}
where $\Xi(t,x)$ is a \emph{saltation matrix} given by
\begin{align}
\nonumber\Xi(t,x)=& \frac{\left(\F_{j'}(t,\R(x))-D\R(x)\cdot \F_{j}(t,x)\right)\cdot Dg_{j,j'} (x)}{Dg_{j,j'}(x)\cdot \F_{j}(t,x)}\\
\label{eq:saltmat}&+D\R(x)
\end{align}
for all $t\geq 0$ and all $x\in \G_{j,j'} = \G{_j}\cap\R^{-1}(\D_{j'})$.

\end{prop}

\section{Main result}
The main contribution of this paper is to provide local conditions under which the distance between any pair of trajectories in a hybrid system
(as measured by the intrinsic metric defined in \eqref{eq:20})
is globally bounded by an exponential envelope. 
These conditions are made precise in Theorem \ref{thm:main} and Corollary \ref{cor:contract}. In the case when the system satisfies a continuous contraction condition within each domain of the hybrid system as well as a discrete nonexpansion condition through the reset map between domains, this exponential envelope is decreasing in time so that the intrinsic distance between trajectories decreases exponentially in time, i.e., the system is contractive.

\begin{example}
\label{ex:1}
Consider a hybrid system with two domains in the positive orthant of the plane so that $\D=\D_{L}\coprod\D_{R}$ with $\D_{L}=\D_{R}=\{x\in\mathbb{R}^2|x_1\geq 0\text{ and }x_2\geq 0\}$, and further take $g _{R,L} (x)=x_1-1$ and $g _{L,R} (x)=1-x_1$ so that the system is in the left (resp., right) domain $\D_{L}$ (resp., $\D_R$) when $x_1< 1$ (resp., $x_1>1$). Assume the reset map $\R$ is the identity map and $\dot{x}=\F_j(x)=A_jx$ for $j\in\{L,R\}$ with
\begin{align}
  \label{eq:54}
  A_j=
  \begin{bmatrix}
    -a_j&0\\
0&-b_j
  \end{bmatrix},\quad a_j,b_j>0\quad \text{for }j\in\{L,R\}.
\end{align}
All executions initialized in $\D$ flow to $\D_L$ and converge to the origin.  
Equip both domains with the standard Euclidean 2--norm so that $|x|_L=|x|_R=|x|_2$ and consider two executions $x(t)=\phi(t,0,\xi)$, $z(t)=\phi(t,0,\zeta)$ with initial conditions $\xi,\zeta\in\D$. 
Then $d(x(t),z(t))=|x(t)-z(t)|_2=|e(t)|_2$ for $e(t)=x(t)-z(t)$. When both executions are in the same domain so that $x,z\in\D_j$ for some $j\in\{L,R\}$, the error dynamics obey the dynamics of that domain. It therefore follows that $D_td(x,z)\leq \max\{-a_j,-b_j\} d(x,z)$
so that the distance decreases at exponential rate $\max\{-a_j,-b_j\} $.

Now suppose $x$ and $z$ are in different domains at some time $t$ and, without loss of generality, assume $x\in \D_L$ and $z\in \D_R$. 
Writing
\begin{align}
  \label{eq:1}
  x=\begin{bmatrix}
1-\epsilon_L\\
x_2
\end{bmatrix},\quad z=
  \begin{bmatrix}
    1+\epsilon_R\\
x_2+\delta
  \end{bmatrix}
\end{align}
for some $\epsilon_L,\epsilon_R>0$ and $\delta\in\mathbb{R}$, we have
\begin{align}
  \label{eq:3}
D_t(d(x,z)^2)&= D_t((x-z)^T(x-z))\\
\nonumber&=2(\epsilon_L+\epsilon_R)(a_L-a_R)+2\delta x_2(b_1-b_2)\\
&\quad +\text{H.O.T.}
\end{align}
where the higher order terms H.O.T. are quadratic in $\epsilon_L$, $\epsilon_R$, and $\delta$. Then $D_t(d(x,z)^2)<0$  for all $x_2\geq 0$ and all sufficiently small $\epsilon_L>0$, $\epsilon_R>0$, $\delta\in\mathbb{R}$ if and only if $a_L<a_R$ and $b_L=b_R$. In other words, contraction between any two arbitrarily close executions transitioning from $\D_R$ to $\D_L$ occurs only if executions ``slow down'' in the direction normal to the guard surface when transitioning domains, and the dynamics orthogonal to the guard are unaffected. This example is illustrated in Figure \ref{fig:ex1}.

\end{example}

\begin{figure}
  \centering
{\footnotesize
    \begin{tikzpicture}
  \begin{axis}[width=2.2in,
    height=1.5in,
   axis y line=left,
   axis x line=bottom,
every axis x label/.style={at={(rel axis cs:1,0)},anchor=west},
every axis y label/.style={at={(rel axis cs:0,1)},anchor=south},
   xmin=.5, xmax=1.5, ymax=1,ymin=0, 
   ytick=\empty,
   xtick={.5,1},
   xticklabels={$1-\epsilon$,1},
    xlabel={$x_1$},
    ylabel={$x_2$},
    clip=true,
    at={(0,0)},
    ]
    \draw[dotted] (1,0) -- (1,.8);
    \node at (.8,.1) {$\D_L$};
    \node at (1.2,.1) {$\D_R$};
    \node[circle, fill=black, draw=black,inner sep=1pt] (a) at (.7,.6) {};
    \node[circle, fill=black, draw=black,inner sep=1pt] (b) at (1.3,.5) {};
    \node[circle, fill=black, draw=black,inner sep=1pt] (a2) at ($(.63, {.41*(exp(1.3*.63)-1)})$) {};
    \node[circle, fill=black, draw=black,inner sep=1pt] (b2) at ($(1.1, {.73*(exp(.4*1.1)-1)})$) {};

    \node[inner sep=1pt,label=90:$x(t)$] (a3) at ($(.77, {.41*(exp(1.3*.77)-1)})$) {};
    \node[,inner sep=1pt,label=90:$z(t)$] (b3) at ($(1.4, {.73*(exp(.4*1.4)-1)})$) {};
    \node[inner sep=1pt] (a4) at ($(.76, {.41*(exp(1.3*.76)-1)})$) {};
    \node[,inner sep=1pt] (b4) at ($(1.39, {.73*(exp(.4*1.39)-1)})$) {};
    \addplot[black,domain=0:.82,dashed] (\x,{.41*(exp(1.3*\x)-1)});
    \addplot[black,domain=0:1.5,dashed] (\x,{.73*(exp(.4*\x)-1)});
    \draw[->,>=latex,line width=1pt] (b) --node[above,pos=.3]{$e(t)$} (a);
    \draw[->,>=latex,line width=1pt] (b2) --node[below=1pt,pos=.8]{$e(t+\tau)$} (a2);
    \draw[->,line width=1pt] (b3)--(b4);
    \draw[->,line width=1pt] (a3)--(a4);
\end{axis}
\end{tikzpicture}
}
\caption{An illustration of two executions $x(t)$ and $z(t)$ of the hybrid system in Example \ref{ex:1} in different domains $\D_L$ and $\D_R$. The distance between executions is the Euclidean length of $e(t)=x(t)-z(t)$. When $x(t)$ and $z(t)$ are close, $|e(t)|$ decreases over a short time window $[t,t+\tau]$ if and only if $a_L<a_R$ and $b_L = b_R$, that is, the horizontal component of $x$ decreases at a slower rate than the horizontal component of $z$ and the rates of change of the vertical components are equal.}
  \label{fig:ex1}
\end{figure}
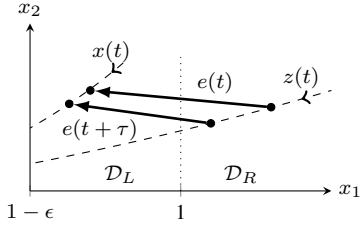

We now generalize the intuition of Example~\ref{ex:1}.
\begin{thm}
\label{thm:main}
Under Assumptions~1--\ref{assum:Rconnect},
if there exists $c\in\mathbb{R}$ such that
  \begin{align}
    \label{eq:thm:mu}
\mu_j\left(D_x \F_j(t,x)\right)\leq c
  \end{align}
for all $j\in \J$, $x\in\D_j\backslash \G_j$, $t\geq 0$,
and
  \begin{align}
    \label{eq:6}
\|\Xi(t,x)\|_{j,j'}\leq 1
  \end{align}
for all $j\in\J$, $x\in\G_{j,j'}$, $t\geq 0$, 
then
  \begin{align}
    \label{eq:thm:Xi}
 d( \phi(t,0,\xi),\phi(t,0,\zeta))\leq e^{ct} d(\xi,\zeta)
  \end{align}
for all $t\geq 0$ and $\xi,\zeta\in \D$.
\end{thm}
\begin{proof}
Given $x(0)=\xi$ and $z(0)=\zeta$, for fixed $\epsilon>0$, let $\gamma:[0,1]\to D$ be a piecewise--differentiable $\R$-connected path satisfying $\gamma(0)=\xi$, $\gamma(1)=\zeta$, and  $L(\gamma)<d(\xi,\zeta)+\epsilon$, and let $\psi(t,r)=\phi(t,0,\gamma(r))$. Since $\phi(t,0,\cdot)$ is piecewise--differentiable, it follows from Assumption \ref{assum:Rconnect} that $\psi(t,\cdot)$ is a piecewise--differentiable $\R$-connected path for all $t\geq 0$. Let $w(t,r)=D_r\psi(t,r)$ whenever the derivative exists. By Proposition 1, $w(t,r)$ satisfies the jump--linear--time-varying equations
\begin{align}
\label{eq:8}
  \dot{w}(t,r)=\frac{\partial f}{\partial x}(t,\psi(t,r))w(t,r),\ \psi(t,r)\in\D\sm\G,
\end{align}
\begin{align}
\label{eq:7}
  w(t^+,r)=\Xi(t,\psi(t^-,r))w(t^-,r),\ \psi(t^-,r)\in \G.
\end{align}

We claim that 
\begin{align}
\label{eq:13}
  |w(t,r)|\leq e^{ct}|w(0^-,r)|
\end{align}
for all $t\geq 0$ and for all $r\in[0,1]$ whenever $w(t,r)$ exists. To prove the claim, for fixed $r$, let $\{t_i\}_{i=1}^k\subset [0,\infty)$ with $t_0\leq t_1\leq \cdots$ and  possibly $k=\infty$ be the set of times at which the hybrid execution $\phi(t,0,\gamma(r))$ intersects a guard so that $\left.\psi(\cdot,r)\right|_{[t_{i},t_{i+1})}$ is continuous for all $i\in\{0,1,\ldots,k-1\}$ where $t_0=0$ by convention, and, additionally, $\left.\psi(\cdot,r)\right|_{[t_{k},\infty)}$ is continuous if $k<\infty$. Note that if $k=\infty$ then $\lim_{i\to\infty} t_i=\infty$ since Zeno executions are not allowed.
Now consider some fixed time $T>0$. If $k<\infty$ and $t_k\leq T$, let $i=k$; otherwise, let $i$ be such that $t_i\leq T <t_{i+1}$. 
Let $j$ be the active domain of the system during the interval $[t_i,t_{i+1})$, i.e. $\psi(t,r)\in \D_j$ for all $t\in[t_i,t_{i+1})$.  
With $J(t)=D_x\F_j(\psi(t,r))$ for $t\in[t_i,t_{i+1})$  
we have
\begin{align}
\label{eq:4}
  |w(T,r)|&\leq e^{\int_{t_{i}}^T\mu(J(\tau))d\tau}|w(t_{i}^+,r)|\\
\label{eq:4-2} &\leq e^{c(T-t_{i})}|w(t_{i}^+,r)|\\
\label{eq:4-3}&\leq e^{c(T-t_{i})}\|\Xi(t,\psi(t_i^-,r))\||w(t_{i}^-,r)|\\
\label{eq:4-4}&\leq e^{c(T-t_{i})}|w(t_{i}^-,r)|
\end{align}
where \eqref{eq:4} follows by Coppel's inequality applied to \eqref{eq:8},  \eqref{eq:4-2} follows from \eqref{eq:thm:mu}, \eqref{eq:4-3} follows from \eqref{eq:7}, and \eqref{eq:4-4} follows from \eqref{eq:6}. Since \eqref{eq:4}--\eqref{eq:4-4} holds for any $T<t_{i+1}$, we further conclude that $|w(t_{i+1}^-,r) |\leq e^{c(t_{i+1}-t_{i})}|w(t_{i}^-,r)|$ whenever $i\leq k$. Then, by recursion, $  |w(T,r)|\leq e^{cT}|w(0^-,r)|$. Since $T$ was arbitrary,~\eqref{eq:13} holds.

Again fix $T> 0$. Because $\psi(T,\cdot)$ is a piecewise--differentiable $\R$-connected path, there exists a finite collection $\{r_i\}_{i=0}^k\subset [0,1]$ with $0=r_0\leq r_1\leq \ldots \leq r_k=1$ such that $\left.\psi(T,\cdot)\right|_{(r_i,r_{i+1})}$ is piecewise--differentiable for all $i\in\{0,1,\ldots, k-2\}$ and $\left.\psi(T,\cdot)\right|_{(r_k-1,r_{k})}$ is piecewise--differentiable. It follows that
\begin{align}
  \label{eq:10}
  L(\left.\psi(T,\cdot)\right|_{[r_i,r_{i+1})})= \int_{r_{i-1}}^{r_i} |w(T,s)| ds.
\end{align}

Then 
\begin{align}
\label{eq:12}
  L(\psi(T,\cdot))&=\sum_{i=0}^{k-1}   L(\left.\psi(T,\cdot)\right|_{[r_i,r_{i+1})})\\
&\leq \int_0^1|w(T,s)|ds\\
\label{eq:12-2}&\leq e^{cT}\int_0^1|w(0^-,s)|ds\\
\label{eq:12-3}&=e^{cT}L(\gamma)\\
&\leq e^{cT}(1+\epsilon)d(\xi,\zeta)
\end{align}
where \eqref{eq:12-2} follows from \eqref{eq:13}, and \eqref{eq:12-3} follows because $w(0^-,r)=D_r\gamma(r)$. In addition, observe
\begin{align}
  \label{eq:14}
 d( \phi(T,0,\xi),\phi(T,0,\zeta))\leq L(\psi(T,\cdot)).
\end{align}
Since $T$ was arbitrary and $\epsilon$ can be chosen arbitrarily small, \eqref{eq:thm:Xi} holds.
\end{proof}
Suppose that global upper and lower bounds on the \emph{dwell time} between successive resets are known. Then, over a given time window, the number of domain transitions is upper and lower bounded, and the proof of Theorem \ref{thm:main} can be adapted to derive an exponential bound on the intrinsic distance between any pair of trajectories
as in the following Corollary.

\begin{corollary}
\label{cor:contract}
Under Assumptions~1--\ref{assum:Rconnect},
suppose the dwell time between resets is at most $\overline{\tau}\in (0,\infty]$ and at least $\underline{\tau}\in[0,\infty)$,
  \begin{equation}
  \mu_j(D_x\F_j(t,x))\leq c
  \end{equation}
for some $c\in\mathbb{R}$ for all $j\in\J$, $x\in D_j\backslash \G_j$, $t\geq 0$, and
\begin{equation}
  	\|\Xi(t,x)\|_{j,j'}\leq K	
\end{equation}
for some $K\in\mathbb{R}_{\geq 0}$ and all $j\in\J$, $x\in\G_{j,j'}$, $t\geq 0$.
Then
\begin{align}
 d( \phi(t,0,\xi),\phi(t,0,\zeta))\leq \max\{K^{\lceil t/\underline{\tau}\rceil},K^{\lfloor t/\overline{\tau}\rfloor}\}e^{ct} d(\xi,\zeta).
\end{align}
In particular, if $\max\{Ke^{c\underline{\tau}},Ke^{c\overline{\tau}}\}<1$ then 
\begin{equation}
\lim_{t\to\infty}d(\phi(t,0,\xi),\phi(t,0,\zeta))=0
\end{equation}
for all $\xi,\zeta\in\D$.
\end{corollary}

We now address an important special case, namely, when domains have the same dimension, are equipped with the same norm, and resets are simple translations (e.g. identity resets). 
Proposition \ref{prop:ident1} establishes that the induced norm of the saltation matrix is lower bounded by unity; 
In the particular case of the standard Euclidean 2--norm, 
Proposition \ref{prop:ident2} shows that 
the induced norm of the saltation matrix is equal to unity if and only if the difference between the vector field evaluated at $x$ and $\R(x)$ lies in the direction of the gradient of the guard function.

\begin{prop}
\label{prop:ident1}
  Under Assumptions \ref{assum:reset}, \ref{assump:vf}, \ref{assump:guard}, and \ref{assump:transverse}, for some $j,j'\in\J$, suppose $D\R_{j,j'}(x)=I$ for all $x\in \G_{j,j'}$, that is, $\R_{j,j'}$ is a translation. Suppose also that $|\cdot|_j=|\cdot|_{j'}$. Then 
  \begin{align}
    \label{eq:16}
\|\Xi(t,x)\|_{j,j'}\geq 1 \text{ for all }x\in\G_{j,j'}\text{, $t\geq 0$}.
  \end{align}
\end{prop}
\begin{proof} Under the hypotheses of the proposition, we have that
  \begin{align}
\Xi(t,x)
= 
I+\frac{\left(\F_{j'}(t,\R(x))-\F_{j}(t,x)\right)\cdot Dg_{j,j'} (x)}{Dg_{j,j'}(x)\cdot \F_{j}(t,x)}.
  \end{align}
Fix $x\in \G_{j,j'}$ and let $z\in\text{Null}(Dg_{j,j'}(x))$. Then $\Xi(t,x) z=z$ so that always $\|\Xi(t,x)\|_{j,j'}\geq 1$.
\end{proof}

\begin{prop}
\label{prop:ident2}
   Under Assumptions \ref{assum:reset}, \ref{assump:vf}, \ref{assump:guard}, and \ref{assump:transverse}, for some $j,j'\in\J$, suppose $D\R_{j,j'}(x)=I$ for all $x\in \G_{j,j'}$,  that is, $\R_{j,j'}$ is a translation. Suppose also that $|\cdot|_j=|\cdot|_{j'}=|\cdot|_2$ where $|\cdot|_2$ denotes the standard Euclidean 2--norm. Then  $\|\Xi(t,x)\|_{j,j'}\leq 1$ if and only if $\|\Xi(t,x)\|_{j,j'}= 1$ if and only if
 \begin{align}
   \label{eq:26}
\F_{j'}(t,\R(x))-\F_{j}(t,x) = \alpha(t,x) Dg_{j,j'} (x)^T
 \end{align}
for all $t\geq 0$ and all $x\in\G_{j,j'}$ for some $\alpha:[0,\infty)\times \G_{j,j'}\to \mathbb{R}$ satisfying 
\begin{align}
  \label{eq:45}
0\leq \alpha(t,x)\leq \frac{-2Dg_{j,j'}(x)\cdot\F(t,x)}{|Dg_{j,j'} (x)|_2^2}.
\end{align}
\end{prop}

\section{Examples}
\renewcommand{\SS}{\texttt{S}}
\newcommand{\CC}{\texttt{C}}

\subsection{A planar piecewise--linear system}
Consider a piecewise--linear system with states in the left-- and right--half plane,
\eqnn{
\D = \D_-\coprod\D_+,\ 
\D_\pm = \set{x = (x_1,x_2)\in\Rbb^2 : \pm x_1 \ge 0},
}
whose continuous dynamics are given by $\dot{x} = A_\pm x,\ x\in\D_\pm$, where
\eqnn{
A_\pm = \mat{cc}{\alpha_\pm & -\beta_\pm \\ \beta_\pm & \alpha_\pm}
}
so that
$\spec A_\pm = \alpha_\pm \pm j\beta_\pm$
and hence the standard Euclidean matrix measure $\mu_2(A_\pm) = \sigma_{\max} \paren{\frac{1}{2} A^\tr_\pm + A_\pm} = \alpha_\pm$.
Supposing $\beta_\pm > 0$, 
all trajectories in $D_\pm$ will eventually reach the set $G_\pm = \set{(x_1,x_2)\in\Rbb^2 : \pm x_1 \le 0, \pm x_2 > 0}$,
where a reset will be applied that scales the second coordinate by $c_\pm > 0$,
\eqnn{
\forall x\in G_\pm\subset\D_\pm : x^+ = \R_\pm(x^-) = (x_1^-,c_\pm x_2^-)\in\D_\mp.
}
This yields a saltation matrix
\eqnn{
\Xi_\pm & = D_x \R_\pm + \paren{\F^\mp - D_x \R_\pm \cdot \F^\pm}\frac{D_x g_\pm}{D_x g_\pm \cdot \F^\pm} \\
& = \mat{cc}{\frac{\beta_\mp}{\beta_\pm} & 0 \\ \frac{1}{\beta_\pm}\paren{\alpha_\pm c_\pm - \alpha_\mp} & c_\pm}.
}

With respect to the standard Euclidean 2--norm:
\begin{enumerate}
\item The continuous--time flows are contractive if $\alpha_\pm < 0$, expansive if $\alpha_\pm > 0$.
\item Unless $A_\pm = A_\mp$ and $\R_\pm = \id_{\Rbb^2}$, one of the discrete--time resets is an expansion.
\end{enumerate}
The first claim follows directly from $\mu_2(A_\pm) = \alpha_\pm$.
To see that the second claim is true, note that 
$\beta_+ \ne \beta_-$, $c_+ > 1$, or $c_- > 1$ implies one of the diagonal entries of one of the $\Xi$'s are expansive.
Taking $\beta_+ = \beta_-$ and $c_\pm \le 1$
to ensure that
the diagonal entries of $\Xi_\pm$ are non--expansive
yields a saltation matrix of the form
\eqnn{
\Xi = \mat{cc}{1 & 0 \\ d & c}
}
with singular values
\eqnn{
\sigma(\Xi) &= \spec\frac{1}{2}\paren{\Xi^\tr + \Xi} \\
&= \frac{1}{2}\paren{(c+1) \pm \sqrt{d^2 + (c-1)^2}};
}
unless $c = 1$ (i.e. $c_+ = c_- = 1$ so $\R_\pm = \id_{\Rbb^2}$) 
and 
$d = 0$ (i.e. $\alpha_+ = \alpha_-$ so $A_+ = A_-$),
one of these singular values is larger than unity.

\subsection{Traffic flow with capacity drop}

  Consider a length of freeway divided into two segments or \emph{links}. The state of the system is the traffic \emph{density} on the two links.  Traffic flows from the first segment to the second. The second link has a finite \emph{jam density} $\xjam_2>0$, and we consider link 1 to have infinite capacity so that always the state $x$ satisfies $x\in\mathcal{X}=[0,\infty)\times [0,\xjam_2]\subset \mathbb{R}^2$.

The system has two modes, an \emph{uncongested} (resp., \emph{congested}) mode for which the flow between the two links depends only on the density of the upstream (resp., downstream) link. The dynamics of the uncongested mode is
  \begin{align}
    \dot{x}_1&=u(t)-\Delta_1(x_1)\\
    \dot{x}_2&= \Delta_1(x_1)-\Delta_2(x_2)
  \end{align}
for which we write $\dot{x}=\F_\text{uncon}(x,t)$ assuming a fixed $u(t)$, and for the congested mode is
\begin{align}
  \label{eq:17}
    \dot{x}_1&=u(t)-S_2(x_2)\\
    \dot{x}_2&=S_2(x_2)-\Delta_2(x_2)  
\end{align}
for which we write $\dot{x}=\F_\text{con}(x,t)$ where $\Delta_1$ and $\Delta_2$ are continuously differentiable and strictly increasing \emph{demand} functions satisfying $\Delta_1(0)=\Delta_2(0)=0$, and $S_2$ is a continuously differentiable and strictly decreasing \emph{supply} function satisfying $S_2(\xjam_2)=0$; see \cite{coogan2015compartmental} for further details of the model.

The system is in the congested mode only (but not necessarily) if $  \Delta_1(x_1)\geq  S_2(x_2)$. 
Moreover, empirical studies suggest 
that traffic flow exhibits a hysteresis effect such that traffic remains in the uncongested mode until $x_2\geq \bar{x}_2$ for some $\bar{x}_2$ and does not return to the uncongested mode until $x_2\leq \ulx_2$ for some $\ulx_2<\bar{x}_2$ \cite{Cassidy:1999vl, Laval:2006zp}. Here, we assume $\bar{x}_2\in [0,\xcrit_2)$ where $\xcrit_2$ is the unique density satisfying 
$\Delta_2(x_2)=S_2(x_2)$; see Figure \ref{fig:traffic2}. This effect is called \emph{capacity drop}.

We model the traffic flow as a hybrid system with four domains $\D_{\SS\CC}, \D_{\SS\bar{\CC}}, \D_{\bar{\SS}\CC}, \D_{\bar{\SS}\bar{\CC}}$ where 
\begin{align}
  \label{eq:49}
  \D_{\SS\CC}&=\mathcal{X}\cap\{x:\Delta_1(x_1)\leq S_2(x_2)\}\cap\{x:x_2\geq \ulx_2\},\\
\D_{\bar{\SS}\CC}&=\mathcal{X}\cap\{x:\Delta_1(x_1)\geq S_2(x_2)\}\cap\{x:x_2\geq \ulx_2\},\\
\D_{\SS\bar{\CC}}&=\mathcal{X}\cap\{x:\Delta_1(x_1)\leq S_2(x_2)\}\cap\{x:x_2\leq \bar{x}_2\},\\
\D_{\bar{\SS}\bar{\CC}}&=\mathcal{X}\cap\{x:\Delta_1(x_1)\geq S_2(x_2)\}\cap\{x:x_2\leq \bar{x}_2\},
\end{align}
and the index set is given by $\mathcal{J}=\{\SS\bar{\CC},\bar{\SS}\bar{\CC},\bar{\SS}\CC,\bar{\SS}\bar{\CC}\}$. Furthermore, $\F_{\SS\CC}= \F_{\SS\bar{\CC}}= \F_{\bar{\SS}\bar{\CC}}=\F_\text{uncon}$ and $\F_{\bar{\SS}\CC}=\F_\text{con}$. As a mnemonic, $\SS$ indicates that $\Delta_1(x_1)\leq S_2(x_2)$ so that adequate downstream supply is available, and $\bar{\SS}$ indicates the opposite. Similarly, $\CC$ indicates the status of the hysteresis effect so that the congestion mode is only possible for domains with $\CC$, and impossible for domains with $\bar{\CC}$.

\begin{figure}
  \centering

   \begin{tikzpicture}[scale=.08]
\fill[gray] (15,30) rectangle (38,38);
\fill[gray] (37.8,32) rectangle (70,38);
\fill[gray] (37.9,30) .. controls +(8,0) and (42,32) .. (50,32) -- (38,38);
\draw[dashed, white] (15,32) -- (40,32);
\draw[dashed, white] (15,34) -- (70,34);
\draw[dashed, white] (15,36) -- (70,36);
  \end{tikzpicture}  

\tikzstyle{link}=[line width=1pt, ->,>=latex, black]
\tikzstyle{junc}=[draw,circle,inner sep=1pt,minimum width=5pt,draw=black]
  \begin{tikzpicture}
\node[junc] (a) at (0,0)  {};
\draw[link,->] (-2,0) -- node[above]{\footnotesize  Link 1} node[below]{\footnotesize $x_1$, density}(a);
\draw[link,->] (a) -- node[above]{\footnotesize  Link 2} node[below]{\footnotesize $x_2$, density} (2,0);
  \end{tikzpicture}
\begin{tabular}{ c c}
\hspace{-10pt}
  \begin{tikzpicture}
\begin{axis}[width=1.8in,
     axis y line=left,
   axis x line=bottom,
every axis x label/.style={at={(rel axis cs:1,0)},anchor=west},
every axis y label/.style={at={(rel axis cs:0,1)},anchor=south},
   xmin=0, xmax=180, ymax=3400,ymin=0, 
   xtick={160},
   xticklabels={$\xjam_1$},
   ytick={\empty},
    xlabel={$x_1$},
    ylabel={flow rate},
    clip=false,
    legend style={at={(.99,.99)},anchor=north east,,cells={align=left}},
    legend cell align=left,
    at={(-260,0)},     ]
    \addplot[name path=A,draw=cfblue,line width=2pt,domain=0:160,samples=50] (\x,{2400*(1-exp(-\x/33))});
    \node[anchor=north] at (160,2400) {\footnotesize $\Delta_1(x_1)$};
\end{axis}
\end{tikzpicture}
&
\hspace{-20pt}
    \begin{tikzpicture}
\begin{axis}[width=1.8in,
     axis y line=left,
   axis x line=bottom,
every axis x label/.style={at={(rel axis cs:1,0)},anchor=west},
every axis y label/.style={at={(rel axis cs:0,1)},anchor=south},
   xmin=0, xmax=180, ymax=3400,ymin=0, 
   xtick={35,65,75,160},
   xticklabels={$\ul{x}_2$,$\bar{x}_2$,, $\xjam_2$},
   ytick={\empty},
    xlabel={$x_2$},
    ylabel={flow rate},
    clip=false,
    legend style={at={(.99,.99)},anchor=north east,,cells={align=left}},
    legend cell align=left,
    at={(-260,0)},     ]
    \addplot[name path=A,draw=cfblue,line width=2pt,domain=0:160,samples=50] (\x,{1900*(1-exp(-\x/33))});
    \addplot[name path=B,draw=dorange,line width=2pt,domain=0:160,samples=2] (\x,{20*(160-x)});
    \path[name path=axis] (axis cs:35,0) -- (axis cs:65,0);
\addplot[fill=dcompb!50]
    fill between[
        of=B and axis,
        soft clip={domain=35:65},
    ];
    \addplot[dashed] coordinates {(75,0) (75,1700)};
    \node[inner sep=0pt] (hyst) at (110,2600) {\footnotesize Hysteresis};
    \draw[->] (hyst.west) to [bend right] +(-25,-650);
    \node[anchor=north] at (160,1900) {\footnotesize $\Delta_2(x_2)$};
    \node[anchor=south] at (160,300) {\footnotesize  $S_2(x_2)$};
    \node[inner sep=0pt, anchor=west] (xcrit) at (90,400) {\footnotesize $\xcrit_2$};
    \draw[->] (xcrit.west) to [bend right] +(-15,0);
\end{axis}
\end{tikzpicture}
\end{tabular}
\caption{Traffic flows from link 1 to link 2. Flow at the interface of link 1 and link 2 depends on the demand $\Delta_1(x_1)$ of link 1 and the supply $S_2(x_2)$ of link 2 and exhibits a hysteresis effect. Traffic exits the network at a flow rate equal to the demand $\Delta_2(x_2)$ of link 2.}
  \label{fig:traffic2}
\end{figure}
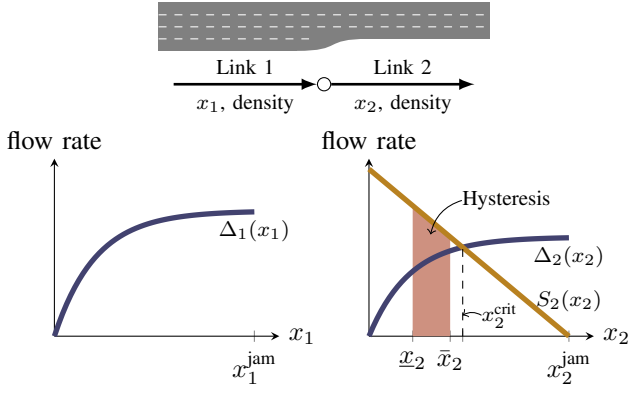

\begin{figure}
  \centering
\begin{tabular}{c c}
\hspace{-12pt}
  \begin{tikzpicture}
  \begin{axis}[width=1.8in,
    height=1.8in,
   axis y line=left,
   axis x line=bottom,
every axis x label/.style={at={(rel axis cs:1,0)},anchor=west},
every axis y label/.style={at={(rel axis cs:0,1)},anchor=south},
   xmin=0, xmax=220, ymax=170,ymin=0, 
   ytick={40,100,160},
   yticklabels={$\ul{x}_2$,$\bar{x}_2$, $\xjam_2$},
   xtick={\empty},
    xlabel={$x_1$},
    ylabel={$x_2$},
    clip=false,
    legend style={at={(.99,.99)},anchor=north east,,cells={align=left}},
    legend cell align=left,
    at={(0,0)}
    ]
    \addplot[name path=A,draw=dorange,line width=2pt,domain=0:200,samples=50] (\x,{-(2200*(1-exp(-\x/33))-3200)/20});
    \addplot[name path=B,draw=cfblue,line width=2pt,domain=0:200,samples=2] (\x,{40});
    \addplot[name path=C,draw=dorange,line width=0pt,domain=0:200,samples=2] (\x,{160});
    \addplot[llblue!60] fill between[of=A and B];
    \addplot[lorange!60] fill between[of=A and C];
    \node[anchor=south] (a) at (130,75) {\footnotesize $\Delta_1(x_1)=S_2(x_2)$};
    \draw [->] ($(a.north west)+(10,-15)$) to [bend right] +(-22,12);
    \node at (100,130) {$\D_{\bar{\SS}\CC}$};
    \node at (25,60) {$\D_{{\SS}\CC}$};
\end{axis}
\end{tikzpicture}
&
\hspace{-20pt}
  \begin{tikzpicture}
  \begin{axis}[width=1.8in,
    height=1.8in,
   axis y line=left,
   axis x line=bottom,
every axis x label/.style={at={(rel axis cs:1,0)},anchor=west},
every axis y label/.style={at={(rel axis cs:0,1)},anchor=south},
   xmin=0, xmax=220, ymax=170,ymin=0, 
   ytick={40,100,160},
   yticklabels={$\ul{x}_2$,$\bar{x}_2$, $\xjam_2$},
   xtick={\empty},
    xlabel={$x_1$},
    ylabel={$x_2$},
    clip=false,
    legend style={at={(.99,.99)},anchor=north east,,cells={align=left}},
    legend cell align=left,
    at={(300,0)},
    ]
    \addplot[name path=A,draw=cfblue,line width=2pt,domain=25:200,samples=50] (\x,{-(2200*(1-exp(-\x/33))-3200)/20});
    \addplot[name path=B,draw=cfblue,line width=0pt,domain=0:200,samples=2] (\x,{00});
    \addplot[name path=C,draw=cfblue,line width=2pt,domain=0:200,samples=2] (\x,{100});
    \addplot[name path=D,draw=cfblue,line width=2pt,domain=0:25,samples=2] (\x,{100});
    \addplot[llblue!60] fill between[of=C and B];
    \addplot[llblue!60] fill between[of=A and C];
    \node[anchor=north] (a) at (85,38) {\footnotesize $\Delta_1(x_1)=S_2(x_2)$};
    \draw [->] ($(a.north east)+(-10,-15)$) to [bend right] +(22,25);
    \node at (150,80) {$\D_{\bar{\SS}\bar{\CC}}$};
    \node at (25,60) {$\D_{{\SS}\bar{\CC}}$};
\end{axis}
\end{tikzpicture}
\end{tabular}
\caption{The traffic network is modeled as a hybrid system with four domains, $\mathcal{J}=\{\SS\bar{\CC},\bar{\SS}\bar{\CC},\bar{\SS}\CC,\bar{\SS}\bar{\CC}\}$. The only non-identity reset occurs when the system transitions from $\D_{\bar{\SS}\bar{\CC}}$ to $\D_{\bar{\SS}\CC}$.}
  \label{fig:traffic}
\end{figure}
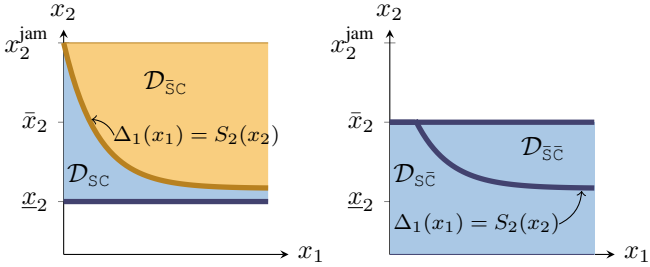

Define the guard functions
\begin{align}
  g_{\SS\CC,\bar{\SS}\CC}(x)&=  g_{\SS\bar{\CC},\bar{\SS}\bar{\CC}}(x)=S_2(x_2)-\Delta_1(x_1),\\
  g_{\bar{\SS}\CC,\SS\CC}(x)&=  g_{\bar{\SS}\bar{\CC},\SS\bar{\CC}}(x)=\Delta_1(x_1)-S_2(x_2),\\
g_{\SS\bar{\CC},\SS\CC}(x)&=g_{\bar{\SS}\bar{\CC},\bar{\SS}\CC}(x)=\bar{x}_2 -x_2,\\
g_{\SS\CC,\SS\bar{\CC}}(x)&=x_2-\ulx_2.
\end{align}

If no guard function is specified between two domains, then no transition is possible between those domains. For all $j,j'\in \mathcal{J}$ such that $g_{j,j'}$ is defined, let $\G_{j,j'}=\{x:g_{j,j'}(x)\leq 0\}\cap \D_j$, and let $\G_j=\cup_{j'\in \J}\G_{j,j'}$ for each $j\in\J$.

We have that
\begin{align}
  \label{eq:28}
  J_\text{uncon}(x)&=D_x\F_\text{uncon}(x,t)=
  \begin{bmatrix}
    -D\Delta_1(x_1)&0\\
D\Delta_1'(x_1)&-D\Delta_2(x_2)
  \end{bmatrix},\\
  J_\text{con}(x)&=D_x\F_\text{con}(x,t)=
                     \begin{bmatrix}
                       0&-DS_2(x_2)\\
                       0&DS_2(x_2)-D\Delta_2(x_2)
                     \end{bmatrix}.
\end{align}

Let $|\cdot|_1$ be the standard one-norm and $\mu_1$ the corresponding matrix measure. It can be verified that
\begin{align}
  \label{eq:30}
\mu_1(  J_\text{uncon}(x))&\leq 0\ \forall x\in\mathcal{X},\quad \text{and}\\
\quad \mu_1(J_\text{con}(x))&\leq 0\ \forall x\in\mathcal{X}.
\end{align}
Now consider a trajectory in domain $\D_{\bar{\SS}\bar{\CC}}$ transitioning to $\D_{\bar{\SS}{\CC}}$ so that $S_2(x_2)\leq \Delta_1(x_1)$ and the system experiences a capacity drop so that the dynamics transition from uncongested to congested. Computing the saltation matrix $\Xi$ for $x$ such that $g_{{\bar{\SS}\bar{\CC}}, {\bar{\SS}{\CC}}}(x)=0$, we have
\begin{align}
  \label{eq:50}
 & \Xi_{{\bar{\SS}\bar{\CC}}, {\bar{\SS}{\CC}}}(t,x)=I+\frac{(\F_\text{con}(t,x)-\F_\text{uncon}(t,x))\cdot D_x g_{{\bar{\SS}\bar{\CC}}, {\bar{\SS}{\CC}}} (x)}{D_x g_{{\bar{\SS}\bar{\CC}}, {\bar{\SS}{\CC}}} (x)\cdot \F_\text{uncon}(t,x)}\\
&=I+  \frac{-1}{\Delta_1(x_1)-\Delta_2(\bar{x}_2)}\begin{bmatrix}
    \Delta_1(x_1)-S_2(\bar{x}_2)\\
S_2(\bar{x}_2)-    \Delta_1(x_1)
  \end{bmatrix}\cdot
  \begin{bmatrix}
    0&-1
  \end{bmatrix}
\end{align}
for all $x\in \{x:x_2=\bar{x}_2\}=\G_{{\bar{\SS}\bar{\CC}}, {\bar{\SS}{\CC}}}(x)$. Let $ \rho(x_1)=\frac{\Delta_1(x_1) -S_2(\bar{x}_2)}{\Delta_1(x_1)-\Delta_2(\bar{x}_2)}$ so that
\begin{align}
  \label{eq:52}
    \Xi_{{\bar{\SS}\bar{\CC}}, {\bar{\SS}{\CC}}}(t,x)=
  \begin{bmatrix}
    1&\rho(x_1)\\
    0&1-\rho(x_1)
  \end{bmatrix}
\end{align}
for all $x\in \{x:x_2=\bar{x}_2\}$. Because $\bar{x}_2<\xcrit_2$, it holds that $\Delta_2(\bar{x}_2)<S_2(\bar{x}_2)$ and therefore
\begin{align}
  \label{eq:47}
0\leq  \rho(x_1)<1\quad \forall x_1\in\{x_1:  \Delta_1(x_1) \geq S_2(\bar{x}_2)\}.
\end{align}
Therefore, $\|    \Xi_{{\bar{\SS}\bar{\CC}}, {\bar{\SS}{\CC}}}(t,x)\|_1=1$ for all $x\in \{x:x_2=\bar{x}_2\}=\G_{{\bar{\SS}\bar{\CC}}, {\bar{\SS}{\CC}}}$.

For all $(j,j')\neq ({{\bar{\SS}\bar{\CC}}, {\bar{\SS}{\CC}}})$ such that $\G_{j,j'}$ is nonempty, it can be verified that $\F_{j'}(x)=\F_j(x)$ for all $x\in \G_{j,j'}$ so that $\Xi_{j,j'}(t,x)=I$ and trivially $\|\Xi_{j,j'}(t,x)\|_1=1$. Applying Theorem \ref{thm:main}, we conclude that
\begin{align}
  \label{eq:25}
  |y(t)-x(t)|_1\leq |y(0)-x(0)|_1
\end{align}
for any pair of trajectories $x(t), y(t)$ of the traffic flow system with initial conditions $y(0),x(0)$ subject to any input $u(t)$, that is, the system is nonexpansive. 

In fact, it is possible to conclude that $\lim_{t\to\infty}|x(t)-\gamma(t)|_1=0$ for any initial condition $x(0)$ using an approach analogous to that used in \cite[Example 4]{Coogan:2016kx}, which considers contraction in traffic flow without modeling capacity drop. In particular, if the derivatives of $\Delta_1$, $\Delta_2$, and $S_2$ are bounded away from zero, and $u(t)$ is periodic with period $T$ and is such that there exists a periodic orbit $\gamma(t)$ of the hybrid system such that $\gamma(t^*)\in\text{int}(\D\backslash\D_{{\bar{\SS}{\CC}}})$ for some $t^*$, then the system is strictly contracting towards $\gamma(t)$ for a portion of each period $T$. This implies that eventually, each trajectory converges to $\gamma(t)$.  Figure \ref{fig:samp} shows simulation results for several initial conditions where $u(t)=1500+800 \cos(2\pi t)$, $\Delta_1(x_1)=2400(1-e^{-x_1/33})$, $\Delta_2(x_2)=1900(1-e^{-x_2/33})$, $S_2(x_2)=20(160-x)$, $\overline{x}_2=60$, $\underline{x}_2=55$, and time is in hours.

\begin{figure}
  \centering
  \begin{tabular}{l}
  \includegraphics[height=1.2in]{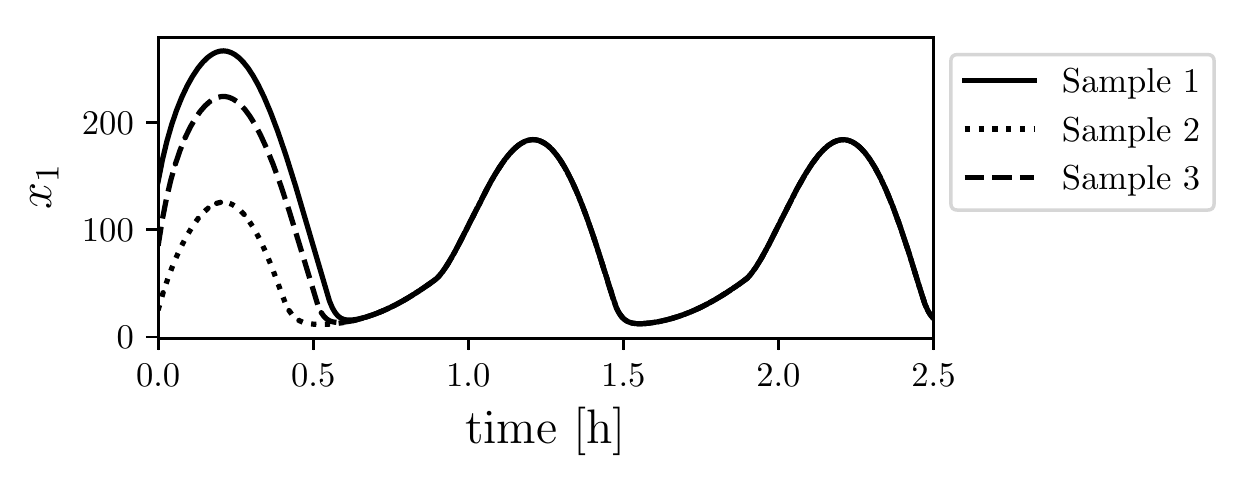}\\
  \includegraphics[height=1.2in]{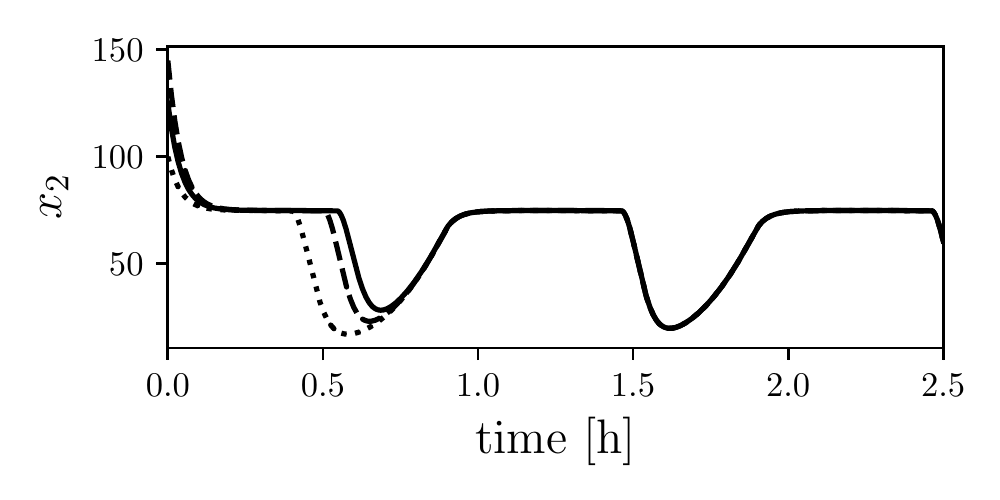}    
  \end{tabular}
  \caption{Sample trajectories of the  two link traffic network with capacity drop modeled as hysteresis with a periodic input. All trajectories contract to a unique periodic trajectory.}
\label{fig:samp}
\end{figure}

\section{Conclusion}
We generalized infinitesimal contraction analysis to hybrid systems by leveraging local dynamical properties of continuous--time flow and discrete--time reset to bound the time rate of change of the intrinsic distance between trajectories.
In addition to expanding the toolkit for analysis of hybrid systems, under dwell time assumptions we provide a novel bound for the intrinsic distance metric.

\bibliographystyle{ieeetr}
\bibliography{burden}

\end{document}